\documentclass[12pt]{amsart}
\usepackage{latexsym,amssymb,amsmath}
\usepackage{amsmath,amsfonts}
\usepackage{epsfig}
\usepackage{graphicx}
\usepackage{verbatim}

\newtheorem{theorem}{Theorem}
\newtheorem{lemma}{Lemma}

\allowdisplaybreaks

\newcommand{\Hmm}[1]{\leavevmode{\marginpar{\tiny%
$\hbox to 0mm{\hspace*{-0.5mm}$\leftarrow$\hss}%
\vcenter{\vrule depth 0.1mm height 0.1mm width \the\marginparwidth}%
\hbox to 0mm{\hss$\rightarrow$\hspace*{-0.5mm}}$\\\relax\raggedright #1}}}
\newcommand{\nc}{\newcommand}
\nc{\les}{\lesssim}
\nc{\ges}{\gtrsim}
\nc{\nit}{\noindent}
\nc{\nn}{\nonumber}
\nc{\D}{\partial}
\nc{\diff}[2]{\frac{d #1}{d #2}}
\nc{\diffn}[3]{\frac{d^{#3} #1}{d {#2}^{#3}}}
\nc{\pdiff}[2]{\frac{\partial #1}{\partial #2}}
\nc{\pdiffn}[3]{\frac{\partial^{#3} #1}{\partial{#2}^{#3}}}
\nc{\abs}[1] {\lvert #1 \rvert}
\nc{\cAc}{{\cal A}_c}
\nc{\cE}{{\cal E}}
\nc{\cF}{{\mathcal F}}
\nc{\cP}{{\cal P}}
\nc{\cV}{{\cal V}}
\nc{\cQ}{{\cal Q}}
\nc{\cGin}{{\cal G}_{\rm in}}
\nc{\cGout}{{\cal G}_{\rm out}}
\nc{\cO}{{\cal O}}
\nc{\Lav}{{\cal L}_{\rm av}}
\nc{\cL}{{\cal L}}
\nc{\cB}{{\cal B}}
\nc{\cZ}{{\cal Z}}
\nc{\cR}{{\cal R}}
\nc{\cT}{{\cal T}}
\nc{\cY}{{\cal Y}}
\nc{\cX}{{\cal X}}
\nc{\cXT}{{{\cal X}(T)}}
\nc{\cBT}{{{\cal B}(T)}}
\nc{\vD}{{\vec \mathcal{D}}}
\nc{\efield}{\mathcal{E}}
\nc{\vE}{{\vec \efield}}
\nc{\vB}{{\vec \mathcal{B}}}
\nc{\vH}{{\vec \mathcal{H}}}
\nc{\ty}{{\tilde y}}
\nc{\tu}{{\tilde u}}
\nc{\tV}{{\tilde V}}
\nc{\Pc}{{\bf P_c}}
\nc{\bx}{{\bf x}}
\nc{\bX}{{\bf X}}
\nc{\bXYZ}{{\bf XYZ}}
\nc{\bY}{{\bf Y}}
\nc{\bF}{{\bf F}}
\nc{\bS}{{\bf S}}
\nc{\dV}{{\delta V}}
\nc{\dE}{{\delta E}}
\nc{\TT}{{\Theta}}
\nc{\dPsi}{{\delta\Psi}}
\nc{\order}{{\cal O}}
\nc{\Rout}{R_{\rm out}}
\nc{\eplus}{e_+}
\nc{\eminus}{e_-}
\nc{\epm}{e_\pm}
\nc{\eps}{\varepsilon}
\nc{\vnabla}{{\vec\nabla}}
\nc{\G}{\Gamma}
\nc{\w}{\omega}
\nc{\mh}{h}
\nc{\mg}{g}
\nc{\vphi}{\varphi}
\nc{\tlambda}{\tilde\lambda}
\nc{\be}{\begin{equation}}
\nc{\ee}{\end{equation}}
\nc{\ba}{\begin{eqnarray}}
\nc{\ea}{\end{eqnarray}}

\nc{\g}{\gamma}
\nc{\ol}{\overline}

\newtheorem{prop}[theorem]{Proposition}

\newtheorem{defin}[theorem]{Definition}

\def\R{\mathbb R}

\nc{\T}{\mathbb T}
\nc{\Z}{\mathbb Z}
\nc{\N}{\mathbb N}
\nc{\pt}{\partial_t}
\nc{\la}{\langle}
\nc{\ra}{\rangle}
\nc{\infint}{\int_{-\infty}^{\infty}}
\nc{\halfwidth}{6.5cm}
\nc{\figwidth}{10cm}
\nc{\nlayers}{L} \nc{\nsectors}{M}
\nc{\indicator}{\mathbf{1}}
\nc{\Rhole}{R_{\rm hole}}
\nc{\Rring}{R_{\rm ring}}
\nc{\neff}{n_{\rm eff}}
\nc{\Frem}{F_{\rm rem}}
\nc{\DD}{\Delta}
\nc{\cD}{\mathcal D}
\nc{\lnorm}{\left\|}
\nc{\rnorm}{\right\|}
\nc{\rnormp}{\right\|_{\ell^{p,\eps}}}
\nc{\rar}{\rightarrow}
\nc{\sgn}{{\rm sign}}
\nc{\non}{\nonumber}
\nc{\wh}{\widehat}
\allowdisplaybreaks
\date{\today}

\begin{document}

\title[Fractional Schr\"odinger equation on torus]{Existence and Uniqueness theory for the fractional Schr\"odinger equation on the torus}

\author[Demirbas, Erdo\u{g}an, Tzirakis]{S. Demirbas, M. B. Erdo\u{g}an, and N. Tzirakis}
\thanks{The first two authors are partially supported by NSF grant  DMS-1201872.}

\address{Department of Mathematics \\
University of Illinois \\
Urbana, IL 61801, U.S.A.}

\email{demirba2@illinois.edu}

\address{Department of Mathematics \\
University of Illinois \\
Urbana, IL 61801, U.S.A.}
\email{berdogan@math.uiuc.edu}

\address{Department of Mathematics \\
University of Illinois \\
Urbana, IL 61801, U.S.A.}
\email{tzirakis@math.uiuc.edu}

\begin{abstract}

We study the Cauchy problem for the $1$-d periodic fractional Schr\"odinger equation with cubic nonlinearity. In particular we prove local well-posedness in Sobolev spaces, for solutions evolving from rough initial data. In addition we show the existence of global-in-time infinite energy solutions. Our tools include a new Strichartz estimate on the torus along with ideas that Bourgain developed in studying the periodic cubic NLS.

\end{abstract}

\maketitle

\section{Introduction}

In this paper we study a fractional semilinear Schr\"odinger type equation with periodic boundary conditions, 

\begin{equation}\label{sch}
\left\{
\begin{array}{l}
iu_{t}+(-\Delta)^{\alpha} u =\pm |u|^2u, \,\,\,\,  x \in {\mathbb{T}}, \,\,\,\,  t\in \mathbb{R} ,\\
u(x,0)=u_0(x)\in H^{s}(\mathbb{T}), \\
\end{array}
\right.
\end{equation}
where $\alpha \in (1/2,1)$. The equation is called defocusing when the sign in front of the nonlinearity is a minus and focusing when the sign is a plus.
\vskip 0.05 in
\noindent
Posed on the real line the equation has appeared at a formal level in many recent articles, see \cite{kay} and the references therein. For example it is a basic model equation in the theory of fractional quantum mechanics introduced by Laskin, \cite{laskin}. A rigorous derivation of the equation can be found in \cite{kay} starting from a family of models describing charge transport in
bio polymers like the DNA. The starting point is a discrete
nonlinear Schr\"odinger equation with general lattice interactions. Equation \eqref{sch} with $\alpha\in(\frac12,1)$ appears as the continuum limit of the long-range interactions between quantum particles on the lattice. Whereas, allowing only the short-range interactions (e.g. neighboring particle interactions) the authors obtain the standard Schr\"odinger equation ($\alpha=1$) which is completely integrable, see \cite{integ}.

In this work we study the periodic problem mainly for two reasons. First due to the lack of strong dispersion the mathematical theory for the fractional Schr\"odinger equations are less developed than the cubic nonlinear Schr\"odinger equation (NLS). Secondly when we consider periodic boundary conditions the analysis becomes harder, for any dispersion relation, since the dispersive character of the equation can only be exploited after employing averaging arguments and a careful analysis of the resonant set of frequencies, \cite {et}.

The local and global well-posedness for the periodic NLS was established by Bourgain in \cite{bou}. He used number theoretic arguments to obtain periodic Strichartz estimates along with a new scale of spaces adapted to the dispersive relation of the linear group. More precisely he proved the existence and uniqueness of local-in-time strong $L^2(\T)$ solutions. Since it is known that smooth solutions of the NLS satisfy mass conservation 
$$M(u)(t)=\int_{\mathbb{T}}|u(t,x)|^2=M(u)(0),$$
Bourgain's result showed the existence of global-in-time strong $L^2(\T)$ solutions in the focusing and defocusing case. The $L^2$ theorem of Bourgain is sharp since as it was shown in \cite{bgt}, the solution operator is not uniformly continuous on $H^s(\Bbb T)$ for $s<0$. 

The local  well-posedness for the  fractional NLS on the real line was recently studied in \cite{kwon}. The authors showed that the equation is locally well-posed in $H^s(\R)$, for $s\geq \frac{1-\alpha}{2}$. They also proved that the solution operator fails to be uniformly continuous in time for $s<\frac{1-\alpha}{2}$. Since the periodic case is less dispersive, we expect the range $s\geq  \frac{1-\alpha}{2}$ to be the optimal range for the local theory also in the periodic case.

In this paper we obtain the following results for the fractional NLS. We first establish a Strichartz estimate that reads as follows
$$\|e^{it(-\Delta)^{\alpha}}f\|_{L^4_{t\in\T}L^4_{x\in\T}}\les \|f\|_{H^s(\T)},$$
for $s> \frac{1-\alpha}{4}$.  
To use this estimate and prove local well-posedness of the equation one has to overcome the derivative loss on the right hand side of the inequality. In principle this can be done by the method in \cite{catwang} and  \cite{sec} which gives local well-posedness in the $H^s(\T)$ level, for $s> \frac{1-\alpha}{2} $. However, since the proof  in \cite{catwang} and \cite{sec} is quite involved, we 
choose to  establish the  local theory  by obtaining  trilinear $X^{s,b}$ estimates directly. Then a standard iteration finishes the proof without any further analysis. We remark that for classical solutions in $H^{s}(\T)$, $s>\frac{1}{2}$, local theory in the space $C([0,T];H^s(\T))$ is known. The proof is the same both on the real line and on the torus and it is based on the Banach algebra property of the Sobolev spaces for $s>\frac{1}{2}$. Moreover the length of the local interval of existence is lower bounded by $\frac{1}{\|u_0\|_{H^s(\T)}^2}$. To lower the regularity of the local existence theory and to prove the smoothing estimate of section 5 we have to reprove the local theory in the $X^{s,b}$ spaces. In this case the solution is controlled on the larger $X^{s,b}$ norm, since $X_T^{s,b}\in C([0,T];H^s(\T))$ for any $b>\frac{1}{2}$, and thus the length of the interval of existence is smaller. In our case for $s>\frac{1}{2}$ it is lower bounded by $\frac{1}{\|u_0\|_{H^s(\T)}^{4+}}$.

We note that in addition to the conservation of mass, smooth solutions of \eqref{sch} satisfy energy conservation  
$$E(u)(t)=\int_{\mathbb{T}}\big||\nabla|^{\alpha} u(t,x)\big|^2\pm \frac{1}{2}\int_{\mathbb{T}}\big|u(t,x)\big|^4=E(u)(0).$$
Note that local theory in $H^\alpha$ level along with the conservation of mass and energy imply the existence of global-in-time energy solutions. Since the equation is mass and energy sub-critical, \cite{tao}, one also obtains global solutions also in the focusing case. This follows from the Gagliardo-Nirenberg inequality
$$\|u\|_{L^4}^4\lesssim \||\nabla|^{\alpha}u\|_{L^2}^{\frac{1}{\alpha}}\|u\|_{L^2}^{4-\frac{1}{\alpha}}$$
which controls the potential energy via the kinetic energy $\||\nabla|^{\alpha}u\|_{L^2}$. One can then control the Sobolev norm of the solution for all times even in the focusing case  since $\frac{1}{\alpha}<2$. We omit the standard details.

In the second part of the paper we use the high-low frequency decomposition of Bourgain, \cite{highlow}, to prove global solutions below the energy level. Bourgain's
method consists of estimating separately the evolution of the low frequencies and of the high
frequencies of the initial data. The low frequency part is smooth and thus by conservation of energy globally defined. The difference equation which is high frequency has small norm. By using smoothing estimates this decomposition can be iterated as long as the norm of the nonlinear part is controlled by the initial energy of the smooth part. As a byproduct of the method one obtains
that the nonlinear part of the solution is actually smoother than the linear propagator and stays always in the energy space. Moreover the global solutions satisfy polynomial-in-time bounds. We summarize the results in the following two theorems:
\begin{theorem}\label{thm:LWP} For any $\alpha\in(\frac12,1)$, and any $b>\frac12$ sufficiently close to $\frac12$, the equation \eqref{sch} is locally well-posed  in the space $X^{s,b}_T \subset C([0,T]; H^s(\T))$  for any  $s>\frac{1-\alpha}2$, where $T=T(\|u_0\|_{H^s(\T)})$. Moreover, for $s>\frac12$ the local existence time   $T\gtrsim \|u_0\|_{H^s(\T)}^{-4-}$.
\end{theorem}

\begin{theorem}\label{thm:GWP}
For any $\alpha\in(\frac12,1)$, the equation \eqref{sch} is globally well-posed in $H^s(\T)$  for any $s>\frac{10 \alpha+1}{12}$. Moreover,
$$u(t)-e^{it(-\Delta)^\alpha\pm iPt}u_0\in H^\alpha(\T)$$
for all times, where $P=\frac1\pi \|u_0\|_2^2$.
\end{theorem}

\noindent
{\bf Remark.}  We will prove Theorem~\ref{thm:GWP} only for the defocusing case. As we mentioned in our introductory remarks since the problem is mass sub-critical for $\alpha>\frac{1}{2}$, we can also control the $H^{\alpha}$ norm of the solution by Gagliardo-Nirenberg inequality in the focusing case. Once we have the control of the norm in terms of the initial energy, the proof of the theorem follows along the same lines. In particular we obtain the same global well-posedness results with the same global-in-time bounds for the focusing problem.
\vskip 0.05in
\noindent

The paper is organized as follows.  In section 2 we introduce our notation and define the spaces that the iteration will take place. In addition we state two elementary lemmas that we use in proving the Strichartz estimates and the multilinear estimates. Section 3 contains the proof of the Strichartz estimate. It is obtained by a careful analysis of the resonant terms and non resonant interacting terms. Section 4 contains the local well-posedness theory for the model equation. We prove multilinear estimates in the $X^{s,b}$ spaces defined in section 2. In section 5 we prove the main smoothing estimate of this paper. The reader should notice that the estimate is sharp within the tools used and for $\alpha=1$ it coincides with the smoothing estimate for the NLS that was recently obtained in \cite{talbot}. Finally in section 6 we use the established local theory and the smoothing estimate to prove global well-posedness for infinite energy solutions. As a final remark we note that our global-in-time results are not optimal.

\section{Notation and Preliminaries}
First of all recall that for $s\geq 0$, $H^s(\T)$ is defined as a subspace of $L^2$ via the norm
$$
\|f\|_{H^s(\T)}:=\sqrt{\sum_{k\in\Z} \la k\ra^{2s} |\widehat{f}(k)|^2},
$$
where $\la k\ra:=(1+k^2)^{1/2}$ and $\widehat{f}(k)=\frac{1}{2\pi}\int_0^{2\pi}f(x)e^{-ikx} dx$ are the Fourier coefficients of $f$. Plancherel's theorem takes the form
$$\sum_{k\in\Z}  |\widehat{f}(k)|^2=\frac{1}{2\pi}\int_{0}^{2\pi}|f(x)|^2dx.$$
We denote the linear propagator of the Schr\"odinger equation as $e^{it(-\Delta)^\alpha}$, where it is defined on the Fourier side as $\widehat{(e^{it(-\Delta)^\alpha}f)}(n)=e^{ it n^{2\alpha}}\widehat{f}(n)$. Similarly, $|\nabla|^{\alpha}$ is defined as $\widehat{|\nabla|^{\alpha}f)}(n)=  n^{ \alpha} \widehat{f}(n)$.
We   also use $(\cdot)^+$ to denote $(\cdot)^\epsilon$ for all $\epsilon>0$ with implicit constants depending on $\epsilon$.

  The Bourgain spaces, $X^{s,b}$, will be defined as the closure of compactly supported smooth functions under the norm $$\|u\|_{X^{s,b}}\dot{=}\|e^{-it(-\Delta)^\alpha}u\|_{H^b_t(\mathbb{R})H^s_x(\mathbb{T})}=\|\langle \tau-|n|^{2\alpha} \rangle^{b} \langle n \rangle^s\widehat{u}(n,\tau)\|_{L_{\tau}^2 l^2_{n}},$$ and the restricted norm will be given as $$\|u\|_{X_T^{s,b}}\dot{=}\inf(\|v\|_{X^{s,b}},\ for\ v=u\ on\ [0,T]).$$

In this paper, by local and global well-posedness we mean the following. 
\begin{defin} We say the equation \eqref{sch} is locally well-posed in $H^s$, if there exist a time $T_{LWP}=T_{LWP}(\|u_0\|_{H^s})$ such that the solution exists and is unique in $X^{s,b}_{T_{LWP}}\subset C([0,T_{LWP}),H^s)$ and depends continuously on the initial data. We say that the the equation is globally well-posed when $T_{LWP}$ can be taken arbitrarily large.
\end{defin} 

We close this section by presenting two elementary lemmas that will be used repeatedly. For the proof of the first lemma see  the Appendix of \cite{erdtzi1}.
\begin{lemma}\label{lem:sums}   If  $\beta\geq \gamma\geq 0$ and $\beta+\gamma>1$, then
\be\nn
\sum_n\frac{1}{\la n-k_1\ra^\beta \la n-k_2\ra^\gamma}\lesssim \la k_1-k_2\ra^{-\gamma} \phi_\beta(k_1-k_2),
\ee
and 
\be\nn
\int_\R \frac{1}{\la \tau-k_1\ra^\beta \la \tau-k_2\ra^\gamma} d\, \tau \lesssim \la k_1-k_2\ra^{-\gamma} \phi_\beta(k_1-k_2),
\ee
where
 \be\nn
\phi_\beta(k):=\sum_{|n|\leq |k|}\frac1{\la n\ra^\beta}\sim \left\{\begin{array}{ll}
1, & \beta>1,\\
\log(1+\la k\ra), &\beta=1,\\
\la k \ra^{1-\beta}, & \beta<1.
 \end{array}\right.
\ee
\end{lemma} 

\begin{lemma}\label{freq_est} Fix $\alpha\in (1/2,1)$. For $n,j,k\in \Z$, we have 
  $$g(j,k,n):=|(n+k)^{2\alpha}-(n+j+k)^{2\alpha}+(n+j)^{2\alpha}-n^{2\alpha}|\ges \frac{|k||j|}{(|k|+|j|+|n|)^{2-2\alpha}},$$
  where the implicit constant depends on $\alpha$.    
\end{lemma}

  \begin{proof}  
 Let $f_c(x)=(x+c)^{2\alpha}-(x-c)^{2\alpha}$. We have
$$
g(j,k,n)=\big|f_{\frac{j}{2}}(n+\frac{j}{2})-f_{\frac{j}{2}}(n+k+\frac{j}{2})\big|.
$$ 
We claim that 
$$f'_c(x)\ges \frac{|c|}{\max(|c|,|x|)^{2-2\alpha}}.$$
Using the claim, we have by the mean value theorem (for $j,k\neq 0$)
\begin{eqnarray}
g(j,k,n)=  \big|f_{\frac{j}{2}}(n+\frac{j}{2})-f_{\frac{j}{2}}(n+k+\frac{j}{2})\big| &\ges & |k||j| \min_{\gamma\in(n+\frac{j}{2},n+k+\frac{j}{2})}\frac{1}{\max(\frac{|j|}{2},|\gamma|)^{2-2\alpha}}\non\\
  &\ges& \frac{|k||j|}{(|k|+|j|+|n|)^{2-2\alpha}}\non.
  \end{eqnarray}
It remains to prove the claim.   
 Since $f_c$ is odd, and $j\neq 0$, it suffices to consider $x\geq 0$ and $c\ges 1$. We have
  $$f'_c(x)=2\alpha\big[(x+c)^{2(\alpha-1)}|x+c|-(x-c)^{2(\alpha-1)}|x-c|\big].$$ We consider three cases:
  
\noindent 
Case 1. $0\leq x\leq c\Rightarrow f'_c(x)=2\alpha\big[(x+c)^{2\alpha-1}+(x-c)^{2\alpha-1}\big]$. Thus $$f'_c(x)\ges c^{2\alpha-1}.$$

\noindent Case 2. $c\leq x\les c\Rightarrow f'_c(x)=2\alpha\big[(x+c)^{2\alpha-1}-(x-c)^{2\alpha-1}\big]$. Then  we get 
$$f'_c(x)\ges c^{2\alpha-1} \Big(\big(\frac{x}{c}+1\big)^{2\alpha-1}-\big(\frac{x}{c}-1\big)^{2\alpha-1}\Big)\gtrsim c^{2\alpha-1} .$$

\noindent Case 3. $x\gg c \Rightarrow f'_c(x)=2\alpha\big[(x+c)^{2\alpha-1}-(x-c)^{2\alpha-1}\big]$. Then we have 
$$f'_c(x)=2\alpha x^{2\alpha-1} \Big(\big(1+\frac{c}{x}\big)^{2\alpha-1}-\big(1-\frac{c}{x}\big)^{2\alpha-1}\Big) \approx  x^{2\alpha-1} \frac{c}{x} = x^{2\alpha-2}c.$$
  
\noindent
Hence, in all cases we have $f'_c(x)\ges \frac{|c|}{\max(|c|,|x|)^{2-2\alpha}}$.

  \end{proof}
  
  \section{Strichartz Estimates}

 \begin{theorem}\label{str}

  $\|e^{it(-\Delta)^{\alpha}}f\|_{L^4_{t\in\T}L^4_{x\in\T}}\les \|f\|_{H^s}$ for $s> \frac{1-\alpha}{4}. $ 

 \end{theorem}
 
 \begin{proof}

Notice that in this proof we can always take $s<\frac14$.  Calling $g=\langle \nabla \rangle^s f$, and denoting $\widehat{g}(k)$ by $g_k$, we write
  
  \begin{align*}
  \|e^{it(-\Delta)^{\alpha}}f\|^4_{L^4_{t}L^4_{x}} & =  \int_0^{2\pi}\int_0^{2\pi}\sum_{k_1,k_2,k_3,k_4}\frac{e^{it(k_1^{2\alpha}-k_2^{2\alpha}+k_3^{2\alpha}-k_4^{2\alpha})}e^{ix(k_1-k_2+k_3-k_4)}g_{k_1}\overline{g_{k_2}}g_{k_3}\overline{g_{k_4}}
}{\langle k_1 \rangle^{s}\langle k_2 \rangle^{s}\langle k_3 \rangle^{s}\langle k_4 \rangle^{s}}dx dt\nonumber\\
  &= \int_0^{2\pi}\sum_{k_1-k_2+k_3-k_4=0}\frac{e^{it(k_1^{2\alpha}-k_2^{2\alpha}+k_3^{2\alpha}-k_4^{2\alpha})}g_{k_1}\overline{g_{k_2}}g_{k_3}\overline{g_{k_4}}
}{\langle k_1 \rangle^{s}\langle k_2 \rangle^{s}\langle k_3 \rangle^{s}\langle k_4 \rangle^{s}}dt \\
 &\les  \sum_{k_1-k_2+k_3-k_4=0  }\frac{|g_{k_1}|| g_{k_2} ||g_{k_3}|| g_{k_4}|}{\langle k_1 \rangle^{s}\langle k_2 \rangle^{s}\langle k_3 \rangle^{s}\langle k_4 \rangle^{s}} \frac{1}{\max(1,|k_1^{2\alpha}-k_2^{2\alpha}+k_3^{2\alpha}-k_4^{2\alpha }|)} 
    \end{align*}
Renaming the variables as $k_1=n+j$, $k_2=n+k+j$, $k_3=n+k$, and $k_4=n$, and using Lemma~\ref{freq_est}, we get
  \begin{align*}
  \|e^{it(-\Delta)^{\alpha}}f\|^4_{L^4_{t}L^4_{x}}&\les 
\sum_{n,k,j}\frac{|g_{n}||g_{n+j}||g_{n+k}||g_{n+k+j}|}
{\langle n \rangle^{s}\langle n+k \rangle^{s}\langle n+j \rangle^{s}\langle n+k+j \rangle^{s}} \frac{1}{\max\big(1,\frac{|kj|}{(|k|+|j|+|n|)^{2-2\alpha}}\big)}\\
&:= I+II
\end{align*}
where $I$ contains the terms with $|kj|\ll(|k|+|j|+|n|)^{2-2\alpha}$ and $II$ contains the remaining terms.
\vskip 0.02in
\noindent
First note that the summation set in $I$ does not contain any terms with both $n=0$ and  $|kj|\neq 0$ since $\alpha\in (1/2,1) $.   
Also noting that if $kj\neq 0$, then 
$$|kj|\ll (|k|+|j|+|n|)^{2-2\alpha}  \les |k|^{2-2\alpha}+|j|^{2-2\alpha}+|n|^{2-2\alpha} \les |kj|+|n|^{2-2\alpha},$$  since $\alpha\in (1/2,1) $. We can thus write 
$$
  I  \les  \sum_{n,k,j\atop 0<|kj|\les |n|^{2-2\alpha}}\frac{|g_{n}||g_{n+j}||g_{n+k}||g_{n+k+j}|}
{\langle n \rangle^{s}\langle n+k \rangle^{s}\langle n+j \rangle^{s}\langle n+k+j \rangle^{s}} +\sum_{j,n}|g_n|^2|g_{n+j}|^2+\sum_{k,n}|g_n|^2|g_{n+k}|^2.
$$
The last two sums are equal to  $\|g\|_{L^2}^4$. We estimate the first sum by Cauchy-Schwarz inequality to get
\begin{align}
& \les  \Big(\sum_{n,k,j}|{g_{n+j}}|^2|g_{n+k}|^2|g_{n+k+j}|^2\Big)^{1/2}\Big(\sum_{n,k,j\atop 0<|kj|\les |n|^{2-2\alpha}}\frac{|g_{n}|^2}{\langle n \rangle^{2s}\langle n+k \rangle^{2s}\langle n+j \rangle^{2s}\langle n+k+j \rangle^{2s}}\Big)^{1/2}\non\\ 
& \les \|g\|_{L^2}^4 \sup_n\Big(\sum_{ k,j\atop 0<|kj|\les |n|^{2-2\alpha}}\frac{1}
{\langle n \rangle^{2s}\langle n+k \rangle^{2s}\langle n+j \rangle^{2s}\langle n+k+j \rangle^{2s}} \Big)^{1/2}.\non
  \end{align}
   The condition on the sum implies, except for finitely many $n$'s, that $|k|\ll |n|$ and $|j|\ll |n|$. Therefore
   \begin{multline*}
 \sum_{  k,j\atop 0<|kj|\les |n|^{2-2\alpha}}\frac{1}
{\langle n \rangle^{2s}\langle n+k \rangle^{2s}\langle n+j \rangle^{2s}\langle n+k+j \rangle^{2s}}\\ \les \frac{1}{\langle n \rangle^{8s}} \sum_{0<|kj|\les |n|^{2-2\alpha}} 1 \les \langle n \rangle^{2-2\alpha -8s}log\langle n \rangle \les 1 
   \end{multline*}
 provided that  $s>\frac{1-\alpha}{4}$.

  \noindent
  For the second sum we have,
  
  \begin{equation*}
  II\les \sum_{n,k,j\atop |kj|\ges |n|^{2-2\alpha}}\frac{|g_{n}||g_{n+j}||g_{n+k}||g_{n+k+j}|(|n|+|k|+|j|)^{2-2\alpha}}
{\langle n \rangle^{s}\langle n+k \rangle^{s}\langle n+j \rangle^{s}\langle n+k+j \rangle^{s}|kj|}.
  \end{equation*}
  
 Using  the symmetry in $k$ and $j$, we have
  \begin{equation*}
  II\les \sum_{n,k,j\atop |kj|\ges |n|^{2-2\alpha},\,|k|\geq |j|}\frac{|g_{n}||g_{n+j}||g_{n+k}||g_{n+k+j}|(|n|+|k|)^{2-2\alpha}}
{\langle n \rangle^{s}\langle n+k \rangle^{s}\langle n+j \rangle^{s}\langle n+k+j \rangle^{s}|kj|}.
  \end{equation*}

\vskip 0.05in
\noindent  
To estimate the sum we consider three frequency regions, $|k|\approx| n|$, $|k|\ll |n|$, and $|k|\gg|n|$.
\vskip 0.05in  
\noindent Region 1. $|k|\approx |n|$. In this region, using Cauchy Schwarz inequality as above, it suffices to show that the sum 
$$
\sum_{|k|\geq |j|\atop |k|\approx |n|}\frac{(|n|+|k|)^{4-4\alpha}}{\la n \ra^{2s} \langle n+k \rangle^{2s}\langle n+j \rangle^{2s}\langle n+k+j \rangle^{2s}k^2 j^2}
$$
is bounded in $n$. We bound this by 
  \begin{align*} 
   \sum_{|k|\geq |j|\atop |k|\approx |n|}\frac{|n|^{2-4\alpha-2s}}{\langle n+k \rangle^{2s}\langle n+j \rangle^{2s}\langle n+k+j \rangle^{2s}j^2}. 
  \end{align*} 
Using the inequality
$$
\la m+j\ra \la j\ra \gtrsim \la m\ra,
$$
and recalling that $s<\frac14$, we obtain
$$
  \les  \sum_{|k|\geq |j|\atop |k|\approx |n|}\frac{|n|^{2-4\alpha-4s}}{\langle n+k \rangle^{4s}j^{2-4s}}\les   \langle n \rangle^{2-4\alpha-4s+1-4s}.
$$  
Here we first summed in $j$ and then in $k$. The sum is bounded in $n$ provided that  $s>\frac{3-4\alpha}{8}$.  
  
\vskip 0.05in  
\noindent Region 2. $|k|\ll |n|$. As in Region 1, it suffices to show that the sum 
$$
\sum_{|j|\leq |k|\ll |n| \atop |kj|\ges|n|^{2-2\alpha}}\frac{|n|^{4-4\alpha}}{\la n \ra^{2s} \langle n+k \rangle^{2s}\langle n+j \rangle^{2s}\langle n+k+j \rangle^{2s}k^2 j^2} \approx \sum_{|j|\leq |k|\ll |n|\atop |kj|\ges|n|^{2-2\alpha}}\frac{|n|^{4-4\alpha-8s}}{ k^2j^2}
$$
  is bounded in $n$. To this end, notice that
  \begin{align*} 
   \sum_{|j|\leq |k|\ll |n|\atop |kj|\ges|n|^{2-2\alpha}}\frac{|n|^{4-4\alpha-8s}}{ k^2j^2} \les  \sum_{|j|\leq |k|\ll |n|}\frac{|n|^{4-4\alpha-8s}}{|j||k|\langle n \rangle^{2-2\alpha}} \les  \sup_n |n|^{2-2\alpha-8s}\log(|n|)^2,\non
  \end{align*}
  which is finite provided that   $s>\frac{1-\alpha}{4}$.   
  
\vskip 0.05in  
\noindent Region 3. $|k|\gg |n|$. In this region we bound the sum by Cauchy Schwarz inequality as follows:

\begin{align*}
& \sum_{|j|\leq |k|,\, |n| \ll |k|  \atop |kj|\ges |n|^{2-2\alpha}}\frac{|g_{n}||g_{n+j}||g_{n+k+j}| |g_{n+k}| |k|^{2-2\alpha}}
{\langle n \rangle^{s}\langle n+k \rangle^{s}\langle n+j \rangle^{s}\langle n+k+j \rangle^{s}|kj|}\\
&\les  \Big(\sum_{n,k,j}|{g_{n }}|^2|g_{n+j}|^2|g_{n+k+j}|^2\Big)^{1/2}\Big(\sum_{|j|\leq |k|,\, |n| \ll |k|   }\frac{|g_{n+k}|^2 |k|^{4-4\alpha}}{\langle n \rangle^{2s}\langle n+k \rangle^{2s}\langle n+j \rangle^{2s}\langle n+k+j \rangle^{2s} k^2j^2 }\Big)^{1/2}\non\\ 
 &\les \|g\|_{L^2}^3  \Big(\sum_{ |j|\leq |k|,\, |n| \ll |k|  }\frac{|g_{n+k}|^2 |k|^{2-4\alpha-2s}}
{\langle n \rangle^{2s} \langle n+j \rangle^{2s}\langle n+k+j \rangle^{2s} j^2}   \Big)^{1/2}.\non
\end{align*}
Estimating the $j$ sum in parenthesis as in Region 1, we have   
  \begin{align*}
 \les  \sum_{|n| \ll |k| }\frac{|g_{n+k}|^2 |k|^{2-4\alpha-2s}}{\langle n \rangle^{4s} \langle n+k  \rangle^{2s} } 
  \les  \sum_{|n| \ll |k| }\frac{|g_{n+k}|^2 |k|^{2-4\alpha-4s}}{\langle n \rangle^{4s}  }  \les  \sum_{n,k  } |g_{n+k}|^2 |k|^{1-2\alpha-4s} \langle n \rangle^{1-2\alpha- 4s}.
  \end{align*}
  We estimate this by Cauchy Schwarz 
  $$
\Big[  \sum_{n,k  } |g_{n+k}|^2 |k|^{2-4\alpha-8s}\Big]^{\frac12} \Big[  \sum_{n,k  } |g_{n+k}|^2 \langle n \rangle^{2-4\alpha- 8s}\Big]^{\frac12} \les \|g\|_{L^2}^2,
  $$
  provided that $2-4\alpha-8s<-1$, i.e. $s>\frac{3}{8}-\frac\alpha{2}$. In the last inequality we summed in $n$ and $k$ separately.
\vskip 0.03in
\noindent 
Thus, for $s>max(\frac{1-\alpha}{4},\frac{3-4\alpha}{8}) =  \frac{1-\alpha}{4} $,   for $\alpha>\frac{1}{2}$, we obtain the Strichartz estimates.
\end{proof}


 \section{Local well-posedness via the $X^{s,b}$ method}

 We will prove Theorem~\ref{thm:LWP} for the defocusing equation by obtaining multilinear estimates in $X^{s,b}$ spaces.  
With the change of variable $u(x,t)\to u(x,t) e^{iPt}$ in the equation \eqref{sch}, where $P=\frac{1}{\pi}\|u_0\|_2^2$, we obtain the equation
$$
iu_t+(-\Delta)^\alpha u+|u|^2u-Pu=0,\,\,\,\,\,\, t\in \R,\,\,\,x\in\T,
$$
with initial data in $u_0\in H^s(\T)$, $s>0$. 

Note the following identity which follows from Plancherel's theorem:
\begin{multline}\label{res_decomp}
\widehat{|u|^2u}(k)=\sum_{k_1,k_2 }   \widehat u(k_1)\overline{\widehat u(k_2)}\widehat u(k-k_1+k_2)\\ =\frac1\pi
\|u\|_2^2\widehat{u}(k)-|\widehat u(k)|^2\widehat{u}(k)+\sum_{k_1\neq k, k_2\neq k_1}  \widehat u(k_1)\overline{\widehat u(k_2)}\widehat u(k-k_1+k_2)\\
=: P  \widehat{u}(k) +\widehat{\rho(u)}(k)+ \widehat{R(u)}(k).
\end{multline} 
Using this in the Duhamel's formula, we have
$$
u(t) =e^{it(-\Delta)^\alpha}u_0(x)-i\int_0^t e^{i(-\Delta)^\alpha (t-t')}(\rho(u)+R(u))dt'.
$$
  By standard iteration techniques, it suffices to obtain an estimate of the form:
$$
\Big\|\int_0^t e^{i(-\Delta)^\alpha (t-t')}(\rho(u)+R(u))dt'\Big\|_{X_T^{s,b}}\les T^\delta \|u\|_{X^{s,b}_T}^3,
$$ 
for $s>\frac{1-\alpha}{2}$ and for some $b>\frac12$, $\delta>0$.
 
To prove this estimate and obtain a lower bound for the local existence time we need the following lemma:

\begin{lemma}\cite{gin} For $b,b'$ such that $0\leq b+b'<1$, $0\leq b'<1/2$, then we have $$\Big\|\int_0^t e^{i(-\Delta)^\alpha(t-\tau)}f(\tau)d\tau\Big\|_{X_T^{s,b}}\lesssim T^{1-b-b'}\|f\|_{X_T^{s,-b'}},$$ for $T\in [0,1].$
\end{lemma}

  \begin{prop}\label{[local]}
  
  Let $\alpha\in(\frac12,1)$ and $s>\frac{1-\alpha}{2}$, then for $b>1/2$ we have,
  $$\big\| \rho(u)+R(u) \big\|_{X^{s,-b'}}\les \|u\|_{X^{s,b}}^3,$$
  provided that $b'<\frac12$ is sufficiently close to $\frac12$. Moreover,  for $s>\frac12$  we can take $b^{\prime}=0$.
  
  \end{prop} 
\vskip 0.05in
\noindent
As we remarked in the introduction, in the case that $s>\frac12$, the condition $b^{\prime}=0$ implies the existence of the local solution in $[0,\delta]$ as long as $\delta^{\frac12-}\|u_{0}\|_{H^s(\T)}^2\sim 1$. This bound although sub-optimal, it is necessary for the proof of the global well-posedness below the energy space that we establish in section 6.  
\begin{proof} We present the proof for $R(u)$. The proof for $\rho(u)$ is easier and in what follows it corresponds to the terms given by $j=k=0$.

First note that $$  \big\|R(u)\big\|_{X^{s,-b'}}= \Big\|\int_{\tau_1-\tau_2+\tau_3=\tau} \sum_{k_1-k_2+k_3=n\atop{k_1\neq n , k_2}}\frac{\widehat{u}(\tau_1,k_1)\overline{\widehat{u}(\tau_2,k_2 )}\widehat{u}(\tau_3,k_3)\langle n \rangle^{s}}{{\langle \tau-n^{2\alpha} \rangle^{b'}}}\Big\|_{L^2_{\tau}l^2_{n}},
$$
By a duality argument and denoting $|\widehat{u}(\tau,n)|\langle n \rangle^{s}\langle \tau-n^{2\alpha} \rangle^{b}=v(\tau,n)$, we get 
 \begin{align*}
   \big\|R(u)\big\|_{X^{s,-b'}}&\leq \sup_{\|g\|_{L^2_\tau l^2_n}=1} \int_{\tau_1-\tau_2+\tau_3-\tau =0} \sum_{ k_1-k_2+k_3-n=0\atop{k_1\neq n , k_2}} \frac{\langle n \rangle^{s}v(\tau_1,k_1)v(\tau_2,k_2)v(\tau_3,k_3)g(\tau ,n)}{\langle k_1 \rangle^{s}\langle k_2 \rangle^{s}\langle k_3 \rangle^{s}\langle \tau -n^{2\alpha} \rangle^{b'}} \\
   & \qquad\qquad\qquad\times \frac{1}{\langle\tau_1-k_1^{2\alpha}\rangle^b\langle\tau_2-k_2^{2\alpha}\rangle^b\langle\tau_3-k_3^{2\alpha}\rangle^b}, 
\end{align*}
and thus, by Cauchy-Schwarz and then integrating in $\tau$ variables as in \cite{erdtzi1}, we have
 \begin{align*} 
   \big\|R(u)\big\|^2_{X^{s,-b'}}&\leq  \|v\|_{L^2_{\tau}l^2_n}^6\sup_{\tau,n}\int_{\tau_1-\tau_2+\tau_3=\tau }\sum_{ k_1-k_2+k_3=n \atop{k_1\neq n , k_2}}\frac{\langle n \rangle^{2s}}{\langle k_1 \rangle^{2s}\langle k_2 \rangle^{2s}\langle k_3 \rangle^{2s}\langle \tau-n^{2\alpha} \rangle^{2b'}} \\
   & \qquad\qquad\qquad\qquad\times \frac{1}{\langle\tau_1-k_1^{2\alpha}\rangle^{2b}\langle\tau_2-k_2^{2\alpha}\rangle^{2b}\langle\tau_3-k_3^{2\alpha}\rangle^{2b}}. \\
   &\les \|u\|_{X^{s,b}}^6 \sup_n \sum_{k_1-k_2+k_3=n\atop{k_1\neq n , k_2}}\frac{\langle n \rangle^{2s}}{\langle k_1 \rangle^{2s}\langle k_2 \rangle^{2s}\langle k_3 \rangle^{2s} \langle k_1^{2\alpha}-k_2^{2\alpha}+k_3^{2\alpha}-n^{2\alpha} \rangle^{2b'}}. 
  \end{align*}
Hence, we need to show that 

$$M_n=\sum_{k_1-k_2+k_3=n\atop{k_1\neq n , k_2}}\frac{\langle n \rangle^{2s}}{\langle k_1 \rangle^{2s}\langle k_2 \rangle^{2s}\langle k_3 \rangle^{2s}\langle k_1^{2\alpha}-k_2^{2\alpha}+k_3^{2\alpha}-n^{2\alpha} \rangle^{2b'}},$$
is bounded in $n$.
Renaming the variables as $k_1=n+j$, $k_2=n+k+j$, $k_3=n+k$,  and using Lemma~\ref{freq_est}, we get
\begin{align*} M_n &\les \sum_{j,k\neq 0}\frac{\langle n \rangle^{2s}}{\langle n+j \rangle^{2s}\langle n+k+j \rangle^{2s}\langle n+k \rangle^{2s} \max\big(1,\frac{|kj|^{2b'}}{(|k|+|j|+|n|)^{4(1-\alpha)b'}}\big)}\\
&:= I+II
\end{align*}
where $I$ contains the terms with $|kj|\ll(|k|+|j|+|n|)^{2-2\alpha}$ and $II$ contains the remaining terms. Here we note that $M_n$ is bounded in $n$ for $b'=0$ in the case $s>\frac12$. From now on we consider the range $\frac{1-\alpha}2<s\leq \frac12$, and take $b'=\frac12-$. To estimate $I$, as in the proof of Theorem~\ref{str}, we write
\begin{align*}
I\les \sum_{0< |kj|\les |n|^{2-2\alpha}} \frac{\langle n \rangle^{2s}}{\langle n+k \rangle^{2s}\langle n+k+j \rangle^{2s}\langle n+j \rangle^{2s} }\les \la n\ra^{2-2\alpha-4s}\log(\la n\ra),
\end{align*}
which is bounded provided that $s>\frac{1-\alpha}{2}$. Similarly,
\begin{align*}
II&\les \sum_{ |kj|\gtrsim |n|^{2-2\alpha}} \frac{\langle n \rangle^{2s} (|k|+|j|+|n|)^{2(1-\alpha)}}{\langle n+k \rangle^{2s}\langle n+k+j \rangle^{2s}\langle n+j \rangle^{2s} |kj|^{1-} } \\
&\les \sum_{ |kj|\gtrsim |n|^{2-2\alpha}\atop{|k|\geq|j|}} \frac{\langle n \rangle^{2s} (|k|+|n|)^{2(1-\alpha)}}{\langle n+k \rangle^{2s}\langle n+k+j \rangle^{2s}\langle n+j \rangle^{2s} |kj|^{1-} }. 
\end{align*}
Second line follows from the $kj$ symmetry of the sum. To estimate the sum we consider three regions:

\noindent
Region 1. $|k|\gg |n|$. The sum is 
  
  \begin{align*} 
  \les   \sum_{|k|\geq|j|\atop |k|\gg |n|}\frac{\langle n \rangle^{2s}|k|^{2(1-\alpha) -2s-1+}}{\langle n+j \rangle^{2s}\langle n+k+j \rangle^{2s}|j|^{1-}}.
  \end{align*}
  Note that for $\frac12 \geq s>\frac{1-\alpha}2$, we can bound it by
   \begin{align*} 
 & \les   \sum_{|k|\geq|j|\atop |k|\gg |n|}\frac{\langle n \rangle^{2s}|k|^{2(1-\alpha) -4s +}}{\langle n+j \rangle^{2s}|k|^{1-2s+}\langle n+k+j \rangle^{2s}|j|^{1-}} \\
  & \les  \sum_{|k|\geq|j|\atop |k|\gg |n|}\frac{\langle n \rangle^{2(1-\alpha)-2s+} }{\langle n+j \rangle^{2s}|k|^{1-2s+}\langle n+k+j \rangle^{2s}|j|^{1-}} \\
  &\les \sum_j \frac{\langle n \rangle^{2(1-\alpha)-2s+} }{\langle n+j \rangle^{2s} |j|^{1-}} \les \langle n \rangle^{2(1-\alpha)-4s+}
  \end{align*}
  which is bounded in $n$. In the $k$ and $j$ sums we used Lemma~\ref{lem:sums}.

\noindent
Region 2. $|k|\approx |n|$. In this region we have the bound
  
  \begin{align*} &\les   \sum_{|k|\geq|j|\atop |k|\approx |n|}\frac{\langle n \rangle^{2s+1-2\alpha+}}{\langle n+j \rangle^{2s}\langle n+k \rangle^{2s}\langle n+k+j \rangle^{2s}|j|^{1-}} \\
  &\les \sum_j \frac{\langle n \rangle^{2s+1-2\alpha+} A}{\langle n+j \rangle^{2s}|j|^{1-}}, 
  \end{align*}
  where $A=|j|^{1-4s}$ if $4s>1$, $A=|n|^{1-4s}$ if $4s<1$ and $A=\log(|n|)$ if $4s=1$. Then,   by considering these cases separately and   using Lemma~\ref{lem:sums} in the $j$ sums, one obtains boundedness in $n$ for $s>\frac{1-\alpha}2$ and $\alpha>\frac12$.

\noindent
Region 3. $|k|\ll |n|$. We have the bound
  
  \begin{align*}
   \les   \sum_{  |j|\leq |k|\ll |n|}\frac{\langle n \rangle^{-4s+2-2\alpha }}{|kj|^{1-}} \les    \langle n \rangle^{-4s+2-2\alpha+},
  \end{align*}
  which is bounded in $n$. 
  \end{proof}

\section{A smoothing estimate}
We first note that
\be\label{rhobound}
\|\rho(u)\|_{H^{s+c}}=\sqrt{\sum_k |\widehat{u}(k)|^6 \la k\ra^{2s+2c}}\les \|u\|_{H^s}^3,
\ee
for $0\leq c\leq 2s$, which implies that the contribution of $\rho(u)$ to the Duhamel formula is smoother than $u$. 
One can also obtain the same level of smoothing in $X^{s,b}$ spaces: For $c\leq 2s$
$$\| \rho(u)\|_{X^{ s+c ,-\frac12+}}\les \|u\|_{X^{s,\frac12+}}^3.$$ 
To prove the same for the non resonant terms $R(u)$ we have the following proposition:
  \begin{prop}\label{[smoothing]} For $s>\frac{1-\alpha}{2}$ and  $c< \min(\alpha-\frac12,2s+\alpha-1)$ , we have $$\| R(u)\|_{X^{ s+c ,-\frac12+}}\les \|u\|_{X^{s,\frac12+}}^3.$$  
  \end{prop}

  \begin{proof}
  Repeating the steps in the proof of Proposition~\ref{[local]},  it suffices to prove that
  $$ M(n)= \sum_{kj\neq 0}\frac{\langle n \rangle^{2s+2c}}{\langle n+j \rangle^{2s}\langle n+k \rangle^{2s}\langle n+j+k \rangle^{2s}\langle \frac{|kj|}{(|n|+|k|+|j|)^{2-2\alpha}} \rangle^{1- }} $$
  is bounded in $n$.

  For the   terms with  $0<|kj|\les|n|^{2-2\alpha}$, since $|k|,|j|\ll |n|$, we have the bound
  
  \begin{equation*}
 \les  \sum_{0<|kj|\les|n|^{2-2\alpha}}\langle n \rangle^{-4s+2c}\les \langle n \rangle^{-4s+2c+2-2\alpha}\log(n),
  \end{equation*}
 which is bounded provided that  $c<2 s+\alpha-1$.  
  
  For the remaining terms, we have to consider the cases  $s>1/2$ and $s\leq 1/2$ separately. Again by symmetry in $j$ and $k$, it is enough to consider $|k|\geq|j|$. 
  
\noindent
Case 1. $s>1/2$. As before, we will consider three regions:
  
 Region 1.1. $|k|\gg|n|$. Then we have
  \begin{align*}
  &\les \sum_{|k|\geq|j|>0 \atop |k|\gg|n|}\frac{\langle n \rangle^{2s+2c}|k|^{1-2\alpha-2s+}}{\langle n+j \rangle^{2s}\langle n+k+j \rangle^{2s}|j|^{1-}} \\
  &\les \sum_{j\atop |k|\gg|n|}\frac{\langle n \rangle^{2c+1-2\alpha+}}{\langle n+j \rangle^{2s}\langle n+k+j \rangle^{2s}|j|^{1-}} \\
  &\les  \sum_{j}\frac{\langle n \rangle^{2c+1-2\alpha +}}{\langle n+j \rangle^{2s}\langle j \rangle^{1-}} 
   \les  \langle n \rangle^{2c -2\alpha +}, 
    \end{align*}
which is bounded  for $c< \alpha $. In the forth inequality we used Lemma~\ref{lem:sums}.
   
    Region 1.2. $|k|\approx |n|$. In this region we have,
   
   \begin{align*}
  \les \sum_{|k|\geq|j|>0 \atop |k|\approx |n|}\frac{\langle n \rangle^{2c+2s+1-2\alpha+}}{\langle n+k \rangle^{2s}\langle n+j \rangle^{2s}\langle n+j+k \rangle^{2s}|j|^{1-}}
   \les \sum_{|k|\geq|j|>0 \atop |k|\approx |n|}\frac{\langle n \rangle^{2c +1-2\alpha+}}{\langle n+k \rangle^{2s} |j|^{1-}}
  \les \langle n \rangle^{2c+1-2\alpha+} 
   \end{align*}
   for $c<\alpha- \frac12$.
   
   Region 1.3. $|k|\ll |n|$. We have
   \begin{align*}
   \les  \sum_{|k|\geq|j|>0\atop |k|\ll |n|}\frac{\langle n \rangle^{-4s+2c+2-2\alpha+}}{|kj|^{1-}} \les \langle n \rangle^{2c-4s+2-2\alpha+}, 
   \end{align*}
   which is bounded  
   for $c<2 s+\alpha-1$. This finishes the case $s>1/2$.
   
\noindent
Case 2. $\frac{1-\alpha}2<s\leq 1/2$.  
   
 Region 2.1. $|k|\gg|n|$. As in the proof of Proposition~\ref{[local]}, we have
   
   \begin{eqnarray}
  \les \sum_{|k|\geq|j|>0 \atop |k|\gg |n|}\frac{\langle n \rangle^{2s+2c-4s+2-2\alpha+}}{\langle n+j \rangle^{2s}\langle n+k+j \rangle^{2s} |k|^{1-2s+} |j|^{1-}} \les  \langle n \rangle^{2c-4s+2-2\alpha+} \non
   \end{eqnarray}
   which is bounded for $c<2s+\alpha-1 $.
   
  Region 2.2. $|k|\approx |n|$. In this region we have,
   
   \begin{eqnarray}\non
  \les  \sum_{|k|\geq|j|>0 \atop |k|\approx |n|}\frac{\langle n \rangle^{2s+2c+1-2\alpha+}}{\langle n+j \rangle^{2s}\langle n+k \rangle^{2s}\langle n+k+j \rangle^{2s}|j|^{1-}}\les  \sum_{j}\frac{\langle n \rangle^{2s+2c+1-2\alpha+} A}{\langle n+j \rangle^{2s}|j|^{1-}}, 
   \end{eqnarray}
   where $A=\langle j \rangle^{1-4s}$ for $\frac14\leq s\leq \frac12$, and $A=\langle n \rangle^{1-4s}$ for $0<s<\frac14$. Hence,
   
   \begin{eqnarray}
   & \les &\langle n \rangle^{2c+1-2\alpha+}\qquad\qquad\quad \text{for $s\geq \frac14$},\non\\
   &\les&  \langle n \rangle^{2c-4s+2-2\alpha+}\qquad\qquad \text{for $0<s<\frac14$}\non, 
   \end{eqnarray}
  which is bounded for $c<2 s+\alpha-1$ when $s\in (0,\frac14)$ and $c<\alpha-\frac12$ when $s\geq\frac14$.  
   
 Region 2.3. $|k|\ll |n|$. We have,
   
   \begin{eqnarray}\non
    \les   \sum_{|k|\geq|j|>0\atop |k|\ll |n|}\frac{\langle n \rangle^{2c+2-2\alpha-4s+}}{|kj|^{1-}} 
    \les   \langle n \rangle^{2c+2-2\alpha-4s+} 
   \end{eqnarray}
   which is bounded for $c<2s+\alpha-1$.
   
   Hence, for all $s$, collecting the results we get the proposition. 
  \end{proof}  
This implies that (see \cite{talbot} for more details):
\begin{theorem} For $\alpha\in(\frac12,1)$, $s>\frac{1-\alpha}2$ and $c<\min(2s+\alpha-1,\alpha-\frac12)$ we have 
$$
\|u(t)-e^{it(-\Delta)^\alpha-iPt}u_0\|_{H^{s+c}}\les \|u_0\|_{H^s}^3
$$
for $t<T$, where $T$ is the local existence time.
\end{theorem}

We finish this section by noting that if we define the multilinear versions of $\rho$ and $R$ via
$$
\widehat{\rho(u,v,w)}(k)= \widehat u(k) \overline{\widehat v(k)} \widehat{w}(k),\,\,\,\,\, 
\widehat{R(u,v,w)}(k)=\sum_{k_1\neq k, k_2\neq k_1}  \widehat u(k_1)\overline{\widehat v(k_2)}\widehat w(k-k_1+k_2),
$$
then the assertions of Proposition~\ref{[local]} and Proposition~\ref{[smoothing]} remain  valid.

   \section{Global Well-posedness via High-Low Frequency Decomposition}
   
From the local theory along with energy and mass conservation, the existence of global solutions in $H^\alpha$ follows easily. In this case, one can control the $H^{\alpha}$ norm and apply the local theory with a uniform in time step to reach any time. In this section we use Bourgain's high-low frequency decomposition together with the smoothing estimate from the previous section to obtain global well-posedness for initial data with infinite energy.
 
   \begin{proof}[Proof of Theorem~\ref{thm:GWP}] Fix $s\in(\frac{1}{2}, \alpha)$.
With the change of variable $u(x,t)\to u(x,t) e^{iPt}$ in   equation \eqref{sch}, where $P=\frac{1}{\pi}\|u_0\|_2^2$, we obtain the equation
$$
iu_t+(-\Delta)^\alpha u+|u|^2u-Pu=0,\,\,\,\,\,\, t\in \R,\,\,\,x\in\T,
$$
with initial data in $u_0\in H^s(\T)$.    In what follows, the implicit constants will depend on $\|u_0\|_{H^s}$. 
We fix $N$ large and decompose the equation 
   into two equations, $u=v+w$:
   \begin{equation}\label{eqn:low}
\left\{
\begin{array}{l}
iv_{t}+(-\Delta)^{\alpha} v +|v|^2v-Pv=0,\\
v(x,0)=P_Nu_0(x)\dot{=}\Phi_0, \\
\end{array}
\right.
\end{equation}
  and
  \begin{equation}\label{eqn:high}
\left\{
\begin{array}{l}
iw_{t}+(-\Delta)^{\alpha} w +|v+w|^2(v+w)-Pw-|v|^2v=0,\\
w(x,0)=u_0(x)-\Phi_0\dot{=}\Psi_0, \\
\end{array}
\right.
\end{equation}
where $P_N$ is the projection onto the frequencies $|n|\leq N$. 

First note that $\|\Phi_0\|_{H^\alpha}\les N^{\alpha-s}$. Moreover, by the local existence theory we presented in $H^\alpha$ and $H^s$ levels, noting that $\alpha>s>\frac12$, we have for $\delta\approx N^{-4(\alpha-s)}$
$$
\|v\|_{X^{\alpha,b}_\delta}\les \|\Phi_0\|_{H^\alpha}\les N^{\alpha-s},\,\,\,\,\,\,\,\,\,\|v\|_{X^{s,b}_\delta} \les \|\Phi_0\|_{H^s}\les 1.
$$
Since equation \eqref{eqn:low} has  the same Hamiltonian, we have
$$
H(v(t))=H(\Phi_0)\les N^{2\alpha-2s}
$$
by the Gagliardo-Nirenberg inequality.

Now pick an $s_0<s$ to be determined later. Note that  $\|\Psi_0\|_{H^{s_0}}\les N^{s_0-s}$. The local existence for $w$ equation follows similarly by the multilinear estimates from the previous sections with the same $\delta$ as above (since the norm of $w$ is small). We thus have
$$
\|w\|_{X^{s_0,b}_\delta}\les \|\Psi_0\|_{H^{s_0}}\les N^{s_0-s},\,\,\,\,\,\,\,\,\,\|w\|_{X^{s,b}_\delta} \les \|\Psi_0\|_{H^s}\les 1.
$$

Now using the decomposition \eqref{res_decomp} for the nonlinearity $\mathcal{N}:= |v+w|^2(v+w)-Pw-|v|^2v$ in \eqref{eqn:high} we have (with $u=v+w$)
 \begin{multline*} 
 \mathcal{N} = Pu-Pw-\frac1\pi \|v\|_{L^2}^2 v +\rho(u)-\rho(v)+R(u)-R(v)\\
 =\frac1\pi \big(\|u_0\|_2^2-\|\Phi_0\|_{L^2}^2\big) v +\rho(u)-\rho(v)+R(u)-R(v).
 \end{multline*}
Using the multilinear smoothing estimate and the multilinearity of $\rho$ and $R$, we have
$$
\|\mathcal N\|_{X^{\alpha,-\frac12+}_\delta} \les \big|\|u_0\|_2^2-\|\Phi_0\|_{L^2}^2\big| \|v\|_{X^{\alpha,-\frac12+}_\delta} + \|w\|_{X^{s_0,b}_\delta}^3+ \|w\|_{X^{s_0,b}_\delta} \|v\|_{X^{s_0,b}_\delta}^2,
$$
for $\alpha-s_0< \min(2s_0+\alpha-1,\alpha-\frac12)$, in particular for $s_0>\frac12$.
 
Ignoring the support condition of $\Phi_0$ and $\Psi_0$, we have 
$$
\big|\|u_0\|_2^2-\|\Phi_0\|_{L^2}^2\big|\les \|\Psi_0\|_{L^2}+\|\Psi_0\|_{L^2}^2\les N^{-s}.
$$
Therefore, we obtain
\begin{multline*}
\|\mathcal N\|_{X^{\alpha,-\frac12+}_\delta} \les N^{-s} \delta^{1-} \|v\|_{X^{\alpha,b}_\delta} + \|w\|_{X^{s_0,b}_\delta}^3+ \|w\|_{X^{s_0,b}_\delta} \|v\|_{X^{\alpha,b}_\delta}^2\\
\les N^{-s} \delta^{1-} N^{\alpha-s}+N^{3(s_0-s)}+N^{s_0-s} N^{2(\alpha-s)} \les N^{2\alpha+s_0 -3s}.
\end{multline*}
Taking $t_1=\delta$, we  write 
$$
u(t_1)=w(t_1)+v(t_1)=e^{it_1(-\Delta)^\alpha+iPt_1}\Psi_0+ w_1(t_1)+v(t_1).
$$
By the bound on $\mathcal N$ and Duhamel's formula, we have 
$$
\|w_1(t_1)\|_{H^{\alpha}}\les N^{2\alpha+s_0 -3s}.
$$
We repeat this process by decomposing $u(t_1)=\Phi_{1}+\Psi_1$, where
$$
\Psi_1=e^{it_1(-\Delta)^\alpha+iPt_1}\Psi_0,\,\,\,\,\,\,\,\,\, \Phi_{1}=w_1(t_1)+v(t_1).
$$
Since $e^{it_1(-\Delta)^\alpha+iPt_1}$ is unitary, $\Psi_1$ satisfies all the properties of $\Psi_0$.
To control the $H^\alpha$ norm of $\Phi_1$, we note 
$$
H(\Phi_1)=H(\Phi_1)-H(v(t_1))+H(v(t_1))=H(w_1(t_1)+v(t_1))-H(v(t_1))+H(\Phi_0),
$$
where the second equality follows from the conservation of the Hamiltonian.

Note that
\begin{multline*}
\big|H(f+g)-H(f)\big| \les \big|\| |\nabla|^\alpha (f+g)\|_2^2 -\| |\nabla|^\alpha f \|_2^2\big| +\int \big| |f+g|^4-|f|^4 \big|\\
\les  \|   g \|_{H^\alpha}^2+ \|   g\|_{H^\alpha}\|   f \|_{H^\alpha} + \int |g| \big(|f|^3+|g|^3\big)\\
\les  \|   g \|_{H^\alpha}^2+ \|   g\|_{H^\alpha}\|   f \|_{H^\alpha} +\|g\|_{H^{\frac14+}}^4 + \|g\|_{H^{\frac14+}} \|f\|_{H^{\frac14+}}^3\\
\les  \|   g \|_{H^\alpha}^2+ \|   g\|_{H^\alpha}\|   f \|_{H^\alpha} +\|g\|_{H^{\alpha}}^4 + \|g\|_{H^{\alpha}} \|f\|_{H^{\alpha}}^3.
\end{multline*} 
Using this for $f=v(t_1)$ and $g=w_1(t_1)$, we obtain
\begin{align*}
  H(w_1(t_1)+v(t_1))-H(v(t_1)) \les  N^{2\alpha+s_0-3s}N^{3(\alpha-s)}=N^{5\alpha+s_0-6s}.
\end{align*}
To reach time $T$ we have to iterate this process $\frac{T}{\delta}$ times. To bound the Hamiltonian at time $T$ by a constant multiple of the initial value, we need
$$
N^{5\alpha+s_0-6s} \frac{T}{\delta}=TN^{9\alpha+s_0-10s}
$$ 
to be $\les N^{2\alpha-2s}$. This holds for  $s>\frac{7\alpha}8+\frac1{16}$  by taking $s_0=\frac12+$ and $N$ sufficiently large.

 The calculation above can be improved by interpolating between $H^\alpha$ and $L^2$ to bound the $H^{\frac14+}$ norms. For example, by Duhamel's formula and Minkowski inequality, we have
$$
\|w_1(t_1)\|_{L^2}\les \int_0^{t_1} \|\mathcal N\|_{L^2} dt.   
$$
The worst term in $\mathcal N$ is of the form $|v^2w|$ which can be bounded as follows
$$
 \delta^{\frac12} \|v\|_{L^4_tL^4_x}^2   \|w\|_{L^\infty_tL^\infty_x} \les  \delta^{\frac12} \|v\|_{L^4_tL^4_x}^2  \|w\|_{X^{s_0,b}_\delta} \les \delta H(v)^{\frac12}  \|w\|_{X^{s_0,b}_\delta} 
 \les \delta N^{\alpha+s_0-2s}.
$$
After, $\frac{T}{\delta}$ steps, the $L^2$ norm remains $\les N^{\alpha+s_0-2s}\les 1$, for $s>\frac\alpha{2}+\frac14$.
Therefore the $L^2$ norm of the low frequency part also remains $\les 1$.

Using this in the bound for the Hamiltonian, we get
\begin{multline*}
  H(w_1(t_1)+v(t_1))-H(v(t_1)) \les  N^{2\alpha+s_0-3s}N^{ \alpha-s }+  N^{(1-\frac1{4\alpha})(\alpha+s_0-2s)} N^{\frac{2\alpha+s_0-3s}{4\alpha}} N^{\frac{3(\alpha-s)}{4\alpha} }N^+\\
  \les N^{3\alpha+\frac12-4s+}+N^{\alpha+\frac32-2s-\frac{s}{\alpha}+}\les N^{3\alpha+\frac12-4s+}.
\end{multline*}
After $\frac{T}{\delta}$ steps we get the bound $  TN^{7\alpha+\frac12-8s+}$. This term is less than similar the initial energy of the high frequency part which is of order $ N^{2\alpha-2s}$ for $s>\frac{5\alpha}6+\frac1{12}.$ We can then iterate our result to reach any time $T$ by sending $N$ to infinity.
\end{proof}

 \end{document}